\documentclass[leqno,11pt]{amsart}
\usepackage[a4paper]{geometry}

\usepackage{times}

\usepackage[all]{xy} 
\usepackage{amsmath, amssymb, amsfonts, latexsym, mdwlist, amsthm}
\usepackage{caption, subcaption}
\usepackage{wrapfig}

\usepackage{url}

\usepackage[bookmarks, colorlinks, breaklinks, pdftitle={A short proof},
pdfauthor={Arend Bayer}]{hyperref}
\hypersetup{linkcolor=blue,citecolor=blue,filecolor=black,urlcolor=blue}

\usepackage{paralist}
\setdefaultenum{(1)}{(a)}{}{}
\usepackage{enumitem} 

\def\C{\ensuremath{\mathbb{C}}}

\def\H{\ensuremath{\mathbb{H}}}

\def\Q{\ensuremath{\mathbb{Q}}}
\def\R{\ensuremath{\mathbb{R}}}
\def\Z{\ensuremath{\mathbb{Z}}}

\def\dim{\mathop{\mathrm{dim}}\nolimits}

\def\inf{\mathop{\mathrm{inf}}\nolimits}

\def\Ext{\mathop{\mathrm{Ext}}\nolimits}
\def\GL{\mathop{\mathrm{GL}}\nolimits}
\def\HN{\mathop{\mathrm{HN}}\nolimits}
\def\Hom{\mathop{\mathrm{Hom}}\nolimits}

\def\Ker{\mathop{\mathrm{Ker}}\nolimits}

\def\rk{\mathop{\mathrm{rk}}}

\def\Stab{\mathop{\mathrm{Stab}}\nolimits}
\def\Slice{\mathop{\mathrm{Slice}}\nolimits}
\def\into{\ensuremath{\hookrightarrow}}
\def\onto{\ensuremath{\twoheadrightarrow}}

\def\blank{\underline{\hphantom{A}}}
\def\wGL{\widetilde{\GL}_2^+(\R)}

\def\abs#1{\left\lvert#1\right\rvert}
\def\norm#1{\left\|#1\right\|}

\newcommand\stv[2]{\left\{#1\,\colon\,#2\right\}}

\makeatletter
\newtheorem*{rep@theorem}{\rep@title}
\newcommand{\newreptheorem}[2]{%
\newenvironment{rep#1}[1]{%
 \def\rep@title{#2 \ref{##1}}%
 \begin{rep@theorem}}%
 {\end{rep@theorem}}}
\makeatother

\newtheorem{Thm}{Theorem}[section]
\newreptheorem{Thm}{Theorem}
\newtheorem{Prop}[Thm]{Proposition}

\newtheorem{Lem}[Thm]{Lemma}

\newtheorem{Cor}[Thm]{Corollary}
\newreptheorem{Cor}{Corollary}

\newreptheorem{Con}{Conjecture}

\newtheorem{Ass}[Thm]{Assumption}

\newtheorem{thm-int}{Theorem}

\theoremstyle{definition}
\newtheorem{Def-s}[Thm]{Definition}
\newtheorem{Def}[Thm]{Definition}
\newtheorem{Rem}[Thm]{Remark}

\def\C{\ensuremath{\mathbb{C}}}

\def\H{\ensuremath{\mathbb{H}}}

\def\Q{\ensuremath{\mathbb{Q}}}
\def\R{\ensuremath{\mathbb{R}}}
\def\Z{\ensuremath{\mathbb{Z}}}

\def\cA{\ensuremath{\mathcal A}}

\def\cD{\ensuremath{\mathcal D}}

\def\cP{\ensuremath{\mathcal P}}

\def\cS{\ensuremath{\mathcal S}}
\def\cQ{\ensuremath{\mathcal Q}}

\def\cZ{\ensuremath{\mathcal Z}}

\usepackage{todonotes}

\begin{document}

\title[A short proof]{A short proof of the deformation property of Bridgeland stability
conditions}

\author{Arend Bayer}
\address{School of Mathematics and Maxwell Institute,
University of Edinburgh,
James Clerk Max\-well Building,
Peter Guthrie Tait Road, Edinburgh, EH9 3FD,
United Kingdom}
\email{arend.bayer@ed.ac.uk}
\urladdr{http://www.maths.ed.ac.uk/~abayer/}

\keywords{Bridgeland stability conditions, Derived category, Wall-crossing}

\begin{abstract}
The key result in the theory of Bridgeland stability conditions is the property that they form a
complex manifold. This comes from the fact that given any small deformation of the central charge, there is a
unique way to correspondingly deform the stability condition.

We give a short direct proof of an effective version of this deformation property.
\end{abstract}


\maketitle
\setcounter{tocdepth}{1}
\tableofcontents

\section{Introduction}

Stability conditions on triangulated categories,  introduced in \cite{Bridgeland:Stab}, have
been hugely influential, due to their connections to physics \cite{Tom-Ivan:quadratic,GMN:WKB}, to
mirror symmetry \cite{Bridgeland:spaces} and to representation theory
\cite{Anno-Bezrukavnikov-Mirkovic:stability}, and due to their applications within algebraic
geometry, for example
to Donaldson-Thomas invariants \cite{Yukinobu:DTsurvey},
to the derived category itself \cite{Daniel:intro-stability, K3Pic1}, or to the birational geometry of moduli spaces
\cite{ABCH:MMP, BM:walls, wallcrossing-BrillNoether, Chunyi-Xiaolei:birational, Emolo-Benjamin:lecture-notes}.

Their distinguishing property, crucial in all applications, is a strong deformation
property: by \cite[Theorem 1.2]{Bridgeland:Stab}, there is a complex manifold of stability
conditions, with a map to a vector space that is a local  isomorphism. We
give a short proof of an effective version of this result.

\subsection*{Result}
We refer to Section \ref{sec:review} for complete definitions; here we 
review notation and the support property. Let $\cD$ be a triangulated category, and let
$v \colon K(\cD) \to \Lambda$ be a homomorphism from its K-group to a finitely generated free
abelian group $\Lambda$. A pre-stability condition on $\cD$ with respect to $v$ is a pair
$\sigma = (Z, \cP)$, where
$\cP$ is a \emph{slicing} (see Definition \ref{def:slicing}) and 
$Z \colon \Lambda \to \C$ is a compatible (see Definition \ref{def:prestability}) group homomorphism.

\begin{Def}[{\cite{Bridgeland:Stab}, \cite{Kontsevich-Soibelman:stability}}]
\label{def:supportproperty}
Let $Q \colon \Lambda_\R \to \R$ be a quadratic form. We say that a pre-stability condition $(Z,
\cP)$ satisfies the support property with respect to $Q$ if 
\begin{enumerate*}
\item \label{item:QKerneg} 
the kernel $\Ker Z \subset \Lambda_\R$ of the central charge is negative definite with respect to $Q$, and
\item \label{item:QEnonneg}
for any semistable object $E$, i.e. $E \in \cP(\phi)$ for some $\phi \in \R$, we have
$Q(v(E)) \ge 0$.
\end{enumerate*}
\end{Def}
In this case, we call $\sigma$ a stability condition. Let $\Stab_\Lambda(\cD)$ denote the
topological space (see Section \ref{sec:review}) of stability conditions, and $\cZ \colon \Stab_\Lambda(\cD) \to \Hom(\Lambda, \C)$ the map
$\cZ(Z, \cP) = Z$. 

\begin{Thm} \label{thm:mainthm}
Let $Q$ be a quadratic form on $\Lambda \otimes \R$, and assume that the stability condition
$\sigma = (Z, \cP)$ satisfies the support property with respect to $Q$. Then:
\begin{enumerate*}
\item \label{item:deformeffective}
There is an open neighbourhood $\sigma \in U_\sigma \subset \Stab_\Lambda(\cD)$ such that the
restriction $\cZ \colon U_\sigma \to \Hom(\Lambda, \C)$ is a covering of the set of 
$Z'$ such that $Q$ is negative definite on $\Ker Z'$.
\item \label{item:Qremains}
All stability conditions in $U_\sigma$ satisfy the support property with respect to $Q$.
\end{enumerate*}
\end{Thm}
In other words, $\Stab_\Lambda(\cD)$ is a manifold, and any path
$Z_t \in \Hom(\Lambda, \C)$ for $t \in [0,1]$ with $Z_0 = Z$ and
$\Ker Z_t$ negative definite for all $t \in [0,1]$ lifts uniquely to a continuous path
$\sigma_t = (Z_t, \cP_t)$ in the space of stability conditions starting at $\sigma_0 = \sigma$.

Part \eqref{item:deformeffective} is an effective variant of \cite[Theorem 1.2]{Bridgeland:Stab}
(which says that there is \emph{some} neighbourhood of $Z_0$ in which paths can be lifted uniquely). The
entire result first appeared as \cite[Proposition A.5]{BMS:stabCY3s} with an indirect proof based on reduction to \cite[Theorem 1.2]{Bridgeland:Stab}.

\subsection*{Remarks} 
The support property can be a deep and interesting property in itself:
a quadratic Bogomolov-Gieseker type inequality for Chern classes of semistable objects which, by
Theorem \ref{thm:mainthm}, is preserved under wall-crossing.

Theorem \ref{thm:mainthm} was crucial in \cite{BMS:stabCY3s} in order to describe an entire
component of the space of stability conditions on abelian threefolds, and on some Calabi-Yau
threefolds. It also greatly simplifies the
construction of stability conditions on surfaces (or of \emph{tilt-stability} on higher-dimensional
varieties \cite{BMT:3folds-BG}). In this case, the quadratic form
$Q$ is the classical Bogomolov-Gieseker inequality, and 
Theorem \ref{thm:mainthm} gives an open subset of stability conditions that
otherwise has to be glued together from many small pieces (see e.g.~\cite[Section 4]{localP2}).

Theorem 1.2 of \cite{Bridgeland:Stab} also allows for components of the space of stability
conditions modelled on a linear subspace $L \subset \Hom(\Lambda, \C)$. When $L$ is defined over $\Q$, we can recover
that statement by replacing $\Lambda$ with $\Lambda/\Ker L$.
(See \cite{Sven-Holger:quotcategories} for examples where this is not satisfied; however, 
to achieve well-behaved wall-crossing one has to assume that $L$ is defined over $\Q$.)

\subsubsection*{Proof idea}
Our proof is based on two ideas. First, we reduce to the case where the imaginary part of $Z$
is constant; then we only have to deform stability in a fixed abelian
category. Secondly, we use the elementary convex geometry of the \emph{Harder-Narasimhan
polygon}, see Section \ref{sec:HNpolygon}.

This avoids the need for \emph{quasi-abelian categories}, and some of the 
more technical arguments of \cite[Section 7]{Bridgeland:Stab}. We still need a few arguments
similar to ones in \cite{Bridgeland:Stab}, which we have reproduced for the convenience of the reader.

\subsection*{Application}
Assume that $\cD$ is a 2-Calabi-Yau category, i.e. for all $E, F \in \cD$ we have a bi-functorial
isomorphism
$\Hom(E, F) = \Hom(F, E[2])^\vee$. Let $\Lambda$ be the \emph{numerical
$K$-group} of $\cD$, and assume that $\Lambda$ is finitely generated. Then
there is a surjection $v \colon K(\cD) \to \Lambda$, and $\Lambda$ admits a non-degenerate bilinear form $(\blank, \blank)$, called Mukai-pairing, with
\[ \chi(E, F) = -\bigl(v(E), v(F) \bigr). \]

Let $\cP_0(\cD) \subset \Hom(\Lambda, \C)$ be the set of central charges $Z$ such that
$\Ker Z$ is negative definite with respect to the Mukai pairing, and such that $\Ker Z$ contains
no root $\delta \in \Lambda, (\delta, \delta) = -2$. 

\begin{Cor}
\label{cor:P0covering}
The restriction
$ \cZ^{-1}\left(\cP_0(\cD)\right) \xrightarrow{\cZ} \cP_0(\cD) $
is a covering map.
\end{Cor}
The proof, given in Section \ref{sec:application}, is fairly similar to the case of
K3 surfaces \cite[Proposition 8.3]{Bridgeland:K3}.
The point of including it here is
to show that in terms of the support property via quadratic forms, and equipped with Theorem
\ref{thm:mainthm}, the proof becomes natural and short.
This result was also proved previously for preprojective algebras of
quivers in \cite{Thomas:stability, Bridgeland:ADE, Ikeda:stability-preprojective}. In each of these
cases, there is a connected component of $\Stab(\cD)$  covering  a connected component of
$\cP_0(X)$; such statements rely crucially on \emph{non-emptiness} of moduli spaces
of stable objects.

\subsection*{Acknowledgements}
I would like to thank Emanuele Macr{\`{\i}} and Paolo Stellari; as indicated above, Theorem
\ref{thm:mainthm} first appeared with a different proof in our joint work \cite{BMS:stabCY3s}. 
I presented similar arguments in my lectures
at the Hausdorff school  on derived categories in Bonn, April 2016;
I am grateful to the organisers for the opportunity, and the participants for their
feedback. I would also like to thank Martin Gulbrandsen and Fran\c cois Charles for pointing out inaccuracies 
in the first arXiv version of this article, and the anonymous referee for detailed comments improving the exposition.
My work was supported by the ERC-StG WallXBirGeom 337039.

\section{Review: definitions and basic properties}
\label{sec:review}

Throughout, $\cD$ will be a triangulated category, equipped with a group homomorphism
\[ v \colon K(\cD) \to \Lambda \]
from its $K$-group to an abelian group $\Lambda \cong \Z^m$.

\subsection*{Definitions}
We first recall the main definitions from \cite{Bridgeland:Stab}.

\begin{Def} \label{def:slicing}
A \emph{slicing} $\cP$ on $\cD$ is a collection of full subcategories $\cP(\phi)$ for all $\phi \in
\R$ with
\begin{enumerate}
\item $\cP(\phi+1) = \cP(\phi)[1]$ for all $\phi \in \R$;
\item for $\phi_1 > \phi_2$ and $E_i \in \cP(\phi_i),  i = 1,2$, we have $\Hom(E_1, E_2) = 0$; and
\item \label{item:HN} 
for any $E \in \cD$ there is a sequence of maps
$ 0 = E_0 \xrightarrow{i_1} E_1 \to \dots \xrightarrow{i_m} E_m = E$
and of real numbers $\phi_1 > \phi_2 > \dots > \phi_m$
such that the cone of $i_j$ is in $\cP(\phi_j)$
for $j = 1, \dots, m$.
\end{enumerate}
\end{Def}
The non-zero objects of $\cP(\phi)$ are called \emph{semistable of phase $\phi$}; its 
simple objects are called \emph{stable}. The sequence of maps in \eqref{item:HN} is called the Harder-Narasimhan (HN)
filtration of $E$.

\begin{Def} \label{def:prestability}
A pre-stability condition on $\cD$ is a pair $\sigma = (Z, \cP)$ where $Z \colon \Lambda \to \C$ is a group homomorphism and $\cP$ a slicing, such that
for all $0 \neq E \in \cP(\phi)$, we have $Z(v(E)) \in \R_{>0}\cdot e^{i \pi \phi}$.
\end{Def}
We will abuse notation and write $Z(E)$ instead of $Z(v(E))$.

\subsection*{Basic properties}
Let $\GL_2^+(\R)$ denote the group of real $2 \times 2$-matrices with positive determinant. Since $\GL_2^+(\R)$ acts on $S^1$, its universal
cover $\wGL$ acts on the universal cover $\R \to S^1$ given explicitly by $\phi \mapsto e^{i\pi\phi}$.
For $\tilde g \in \wGL$ we will write $g$  for the corresponding element of $\GL_2^+(\R)$, and
$\tilde g.\phi$ for the given action on $\R$.

\begin{Prop} \label{prop:wGL}
There is a natural action of $\wGL$ on the set of pre-stability conditions given by
$\tilde g.(Z, \cP) = (g \circ Z, \cP')$ where
$\cP'(\tilde g.\phi) = \cP(\phi)$.
\end{Prop}

The \emph{heart of a bounded t-structure} is a full subcategory $\cA \subset \cD$ such that
\[ 
\cS(\phi) := \begin{cases} \cA[\phi] & \text{if $\phi \in \Z$} \\
			0 & \text{if $\phi \notin \Z$} \end{cases}
\]
is a slicing (see \cite[Lemma 3.2]{Bridgeland:Stab}). It is automatically an abelian subcategory; 
and stability conditions on $\cD$ can be constructed from slope-stability in $\cA$.

\begin{Def}
A stability function $Z$ on an abelian category $\cA$ is a morphism
$Z \colon K(\cA) \to \C$ of abelian groups such that for all $0 \neq E \in \cA$, the complex number 
$Z(E)$ is in the semiclosed upper half plane
\[  \H:= \stv{z \in \C}{\Im z > 0\  \text{or} \ z \in \R_{<0}}.
\]
\end{Def}

For $ 0 \neq E \in \cA$ we define its phase by $\phi(E) := \frac 1\pi \arg Z(E) \in (0,1]$.
An object $E \in \cA$ is called $Z$-semistable if for all
subobjects $A \into E$, we have $\phi(A) \le \phi(E)$. 

\begin{Def} We say that a stability function $Z$ on an abelian category $\cA$ satisfies the
\emph{HN property} if every object $E \in \cA$ admits a Harder-Narasimhan (HN) filtration: a
sequence $0 = E_0 \into E_1 \into E_2 \into \dots \into E_m = E$ such that
$E_i/E_{i-1}$ is $Z$-semistable for $i = 1, \dots, m$, with
\[ \phi\left(E_1/E_0\right) > \phi\left(E_2/E_1\right) > \dots > \phi\left(E_m/E_{m-1}\right). \]
\end{Def}

\begin{Prop}[{\cite[Proposition 5.3]{Bridgeland:Stab}}] \label{prop:stabviaheart}
To give a pre-stability condition on $\cD$ is equivalent to giving a heart $\cA$ of a bounded
t-structure, and a stability function $Z$ on $\cA$ with the HN property.
\end{Prop}
Here we tacitly assume that $Z$ factors via
$K(\cA) = K(\cD) \xrightarrow{v} \Lambda$. Given $(Z, \cA)$, the slicing is determined by
setting $\cP(\phi)$ to be the $Z$-semistable objects $E \in \cA$ of phase $\phi$ for $\phi \in (0,
1]$. Conversely,
given  $(Z, \cP)$, the heart $\cA = \cP(0, 1]$ is the smallest extension-closed subcategory of $\cD$ containing
$\cP(\phi)$ for $\phi \in (0, 1]$. More generally, $\cP(\phi, \phi+1]$ is a heart for every $\phi \in \R$.

\begin{Def} A stability condition $\sigma$ is a pre-stability condition that satisfies the support
property in the sense of Definition \ref{def:supportproperty} with respect to some quadratic
form $Q$ on $\Lambda \otimes \R$.
\end{Def}

\subsection*{Topology and local injectivity} 
There is a generalised metric, and thus a topology, on the set of slicings $\Slice(\cD)$ given as follows. Given two slicings $\cP, \cQ$, we write
$\phi^{\pm}(E)$ and $\psi^{\pm}(E)$ for the largest and smallest phase in the associated HN
filtration of an object $E$ for $\cP$ and $\cQ$, respectively. Then we define the distance of $\cP$ and $\cQ$ by
\[
d(\cP, \cQ) := \sup \stv{ \abs{\phi^+(E) - \psi^+(E)}, \abs{\phi^-(E) - \psi^-(E)}}{E \in \cD} \in [0, +\infty].
\]

We recall that this distance can be computed by considering $\cP$-semistable objects alone:
\begin{Lem}[{\cite[Lemma 6.1]{Bridgeland:Stab}}] \label{lem:distviasemistables}
We have $d(\cP, \cQ) = d'(\cP, \cQ)$, where the latter is defined by
\[ d'(\cP, \cQ) := \sup \stv{\psi^+(E) - \phi, \phi - \psi^-(E)}
{\phi \in \R, 0 \neq E \in \cP(\phi)}.\]
\end{Lem}
\begin{proof}
The inequality $d(\cP, \cQ) \ge d'(\cP, \cQ)$ is immediate. For the converse,
consider $E \in \cD$, and let $A_i$ be one of its HN factors with respect to $\cP$.
Then $\psi^+(A_i) \le \phi(A_i) + d'(\cP, \cQ) \le \phi^+(E) + d'(\cP, \cQ)$. 
Hence $E$ admits no non-zero maps from $\cQ$-stable objects
of phase bigger than $\phi^+(E) + d'(\cP, \cQ)$, and so $\psi^+(E) \le \phi^+(E) + d'(\cP, \cQ)$. 
The analogous inequality $\psi^-(E) \ge \phi^-(E) - d'(\cP, \cQ)$ follows similarly.

Finally, for $E \in \cD$ we have a non-zero map $A_1 \to E$ with $A_1 \in \cP(\phi^+(E))$.
Therefore, $\psi^+(E) \ge \psi^-(A_1) \ge \phi^-(A_1) - d'(\cP, \cQ) = \phi^+(E) - d'(\cP, \cQ)$,
and so $\abs{\psi^+(E) - \phi^+(E)} \le d'(\cP, \cQ)$. Combined with the same inequality for
$\psi^-$, the claim follows.
\end{proof}

The topology on $\Stab_{\Lambda}(\cD)$ 
is the coarsest topology such that both forgetful maps
\begin{align*}
\Stab_{\Lambda}(\cD) & \to \Slice(\cD),  \quad (Z, \cP) \mapsto \cP \\
\cZ \colon \Stab_{\Lambda}(\cD) & \to \Hom(\Lambda, \C),  \quad (Z, \cP) \mapsto Z 
\end{align*}
are continuous.

The following Lemma is a variant of \cite[Lemma 6.4]{Bridgeland:Stab}:

\begin{Lem} \label{lem:generaldistanceboundfromZ}
	Assume that $\sigma = (Z, \cP)$ and $\tau = (W, \cQ)$ are two pre-stability conditions such that $\sigma$ satisfies the support property with respect to $Q$, such that
	\begin{equation} \label{ineq:ZW}
	\frac{\abs{W(v)-Z(v)}}{\abs{Z(v)}} < \sin \pi \epsilon \quad \text{for all $v \in \Lambda$ with $Q(v) \ge 0$,}
	\end{equation}
	and such that either $d(\cP, \cQ) < \frac 14$, or that $\sigma, \tau$ have the same heart
	$\cP(0,1] = \cQ(0,1]$. Then $d(\cP, \cQ) < \epsilon$.
\end{Lem}
Of course, \eqref{ineq:ZW} means in particular that the phases of $W(v)$ and $Z(v)$ differ by at
most $\epsilon$.
\begin{proof}
	We want to apply Lemma \ref{lem:distviasemistables}, so let us consider an object $E \in \cP(\phi)$. In the first case, $d(\cP, \cQ) < \frac 14$, we apply this assumption twice to see that
	every HN filtration factor of $E$ with respect to $\cQ$ is contained in 
	$\cQ(\phi - \frac 14, \phi + \frac 14) \subset \cP(\phi-\frac 12, \phi+\frac 12] =: \cA$. In particular, the first
	HN filtration factor $E_1 \to E$ of $E$ with respect to $\cQ$ is a subobject of $E$ in the abelian category
	$A$. The analogous claim is obvious in the second case for $\cA = \cP(0,1]$. Therefore, every HN factor $F$ of $E_1$ with respect to $\cP$ has phase at most $\phi$. By \eqref{ineq:ZW}, it follows
	that $W(F)$ has phase less than $\phi + \epsilon$, and thus the same holds for the phase $\psi^+(E)$ of $W(E_1)$. 
	
	A similar argument shows $\psi^-(E) > \phi(E) - \epsilon$, thus proving the claim.
\end{proof}

\begin{Cor} \label{cor:continuoussection}
The  map $\cZ \colon \Stab_{\Lambda}(\cD) \to \Hom(\Lambda, \C)$, see Definition \ref{def:slicing}, is locally injective.

Moreover, consider a section $U \to \Stab_\Lambda(\cD), Z \mapsto \sigma_Z = (Z, \cP_Z)$ of $\cZ$ defined on a subset
$U \subset \Hom(\Lambda, \C)$, such that every $\sigma_Z$ satisfies the support property with
respect to $Q$.  Assume that $U$ can be covered by open subsets $V_i$ such that for
all $Z, Z' \in V_i$ we have $d(\cP_Z, \cP_{Z'}) < \frac 14$. Then this section is continuous.
\end{Cor}
\begin{proof}
For the first statement, we just set $Z = W$ in the Lemma, to obtain $d(\cP, \cQ) = 0$.

For the second statement, we only need to show that \eqref{ineq:ZW} holds in a neighbourhood
of $Z \in \Hom(\Lambda, \C)$ when $Q$ is negative definite on $\Ker Z$. Choose any metric on
$\Lambda_\R$; then clearly \eqref{ineq:ZW} only needs to be checked for vectors with unit length.
Since $Q(\blank) \ge 0$ defines a compact subset of the unit sphere, on which $\abs{Z(v)}$ is a
positive continuous function, the claim follows.
\end{proof}

Morever, Proposition \ref{prop:wGL} gives a \emph{continuous} action of $\wGL$ 
on $\Stab_\Lambda(\cD)$.

\section{Harder-Narasimhan filtrations via the Harder-Narasimhan polygon}\label{sec:HNpolygon}

Throughout this section, let $\cA$ be an abelian category with a stability function $Z$.

\begin{Def}
The \emph{Harder-Narasimhan polygon} $\HN^Z(E)$ of an object $E \in \cA$ is the convex
hull of the central charges $Z(A)$ of all subobjects $A \subset E$ of $E$.
\end{Def}
(The trivial subobjects $A = 0$ or $A = E$ are included in the definition.) The idea to consider
this convex set in the context of slope-stability goes back at least 40 years
\cite{Shatz:Degeneration}.

\begin{Def} We say that the Harder-Narasimhan polygon $\HN^Z(E)$ of an object $E \in \cA$
is \emph{polyhedral on the left} if the set has finitely many extremal points
$0 = z_0, z_1, \dots, z_m = Z(E)$ such that $\HN^Z(E)$ lies to the right\footnote{By ``to the right of'' a given path $\gamma$ we mean all points of the form $z + x$ where
	$z \in \gamma$ and $x \in \R_{\ge 0}$.}  of the 
path $z_0z_1z_2\dots z_m$;
see fig.~\ref{fig:HNpoly}.
\end{Def}

	\begin{wrapfigure}[8]{r}{0.4\textwidth}
 \centering
        \includegraphics{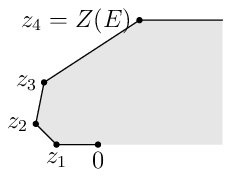}
    \caption{Polyhedral on the left}
    \label{fig:HNpoly}
\end{wrapfigure}

In other words, the intersection of $\HN^Z(E)$ with the closed half-plane to the
left of the line through $0$ and $Z(E)$ is the polygon with vertices $z_0, z_1, \dots, z_m$.
Our proof of Theorem \ref{thm:mainthm} is based on the following well-known statement; we provide a
proof for completeness:

\begin{Prop} \label{prop:HNviapoly}
The object $E$ has a Harder-Narasimhan filtration with respect to $Z$ if and only if its
Harder-Narasimhan polygon $\HN^Z(E)$ is polyhedral on the left.
\end{Prop}

Assume that $\HN^Z(E)$ is polyhedral on the left. For each $i = 1, \dots, m$, choose a subobject $E_i \subset E$ such that $Z(E_i) = z_i$. (This exists as $z_i$ is extremal.)
\begin{Lem}  \label{lem:filtration}
This is a filtration, i.e. $E_{i-1} \subset E_{i}$ for $i = 1, \dots, m$. 
\end{Lem}
\begin{proof}
Let $A := E_{i-1} \cap E_i \subset E$ be the intersection of two subsequent objects, and
$B := E_{i-1} + E_i \subset E$ be their span inside $E$. Then there is a short exact sequence
\[ A \into E_{i-1} \oplus E_i \onto B. \]
Hence the midpoint of $Z(A)$ and $Z(B)$ is also the midpoint of $z_{i-1}$ and $z_i$,
see also figure \ref{fig:HNproofdetail}.

\begin{figure}[!htb]
\centering
\begin{minipage}{0.45\textwidth}
	\centering
        \includegraphics{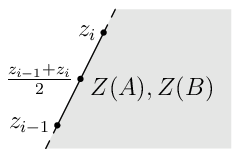}
    \captionof{figure}{Lemma \ref{lem:filtration}}
    \label{fig:HNproofdetail}
\end{minipage}
\begin{minipage}{0.45\textwidth}
	\centering
        \includegraphics{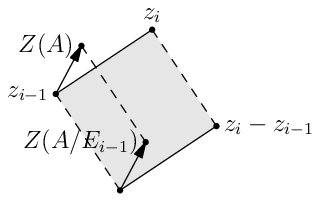}
    \captionof{figure}{Lemma \ref{lem:quotsemistable}}
    \label{fig:HNproofdetail2}
\end{minipage}
\end{figure}

On the other hand, $Z(A), Z(B)$ lie in $\HN^Z(E)$; by convexity and the choice of $z_{i-1}, z_i$,
they both have to lie either in the open half-plane to the right of the line
$\left(z_{i-1}z_i\right)$, or on the line segment $\overline{z_{i-1}z_i}$. The former would be a
contradiction to the previous paragraph, and so $Z(A) \in \overline{z_{i-1}z_i}$. 

Since $A \subset E_{i-1}$, this implies $Z(A) = z_{i-1}$ and $A \cong E_{i-1}$; therefore, 
$E_{i-1} \subset E_i$.
\end{proof}

\begin{Lem} \label{lem:quotsemistable}
The filtration quotient $E_i/E_{i-1}$ is semistable. 
\end{Lem}
\begin{proof}
Otherwise, there is an object $A$ with $E_{i-1} \subset A \subset E_i$ such that
$A/E_{i-1}$ has bigger phase than $E_i/E_{i-1}$, see fig.~\ref{fig:HNproofdetail2}. It follows that $Z(A)$ lies to the left of the line
segment $\overline{z_{i-1}z_i}$. Since $A \subset E$ and hence $Z(A) \in \HN^Z(E)$, this is a
contradiction.
\end{proof}

\begin{proof}[Proof of Proposition \ref{prop:HNviapoly}]
The phase of $E_i/E_{i-1}$ is determined by the argument of $z_i - z_{i-1}$;
by convexity this shows $\phi(E_1/E_0) > \dots > \phi(E_m/E_{m-1})$, and so the
$E_i$ form a HN filtration.

Conversely, assume that we are given a HN filtration $0 = E_0 \into E_1 \into \dots \into E_m$
and a subobject $A \into E$.
We have to show that $Z(A)$ lies to the right of the 
path $z_0z_1 \dots z_m$ with vertices $z_i := Z(E_i)$. By induction on $m$, we may 
assume that $Z(A \cap E_{m-1})$ lies to the right of the path $z_0z_1 \dots z_{m-1}$. On the other hand,
$A/\left(A \cap E_{m-1}\right)$ is a subobject of $E_m/E_{m-1}$, which is semistable; thus its central charge
 $Z\bigl(A/\left(A \cap E_{m-1}\right)\bigr)$ lies to the right of the line segment from $0$ to $z_m - z_{m-1}$. 
Therefore, $Z(A) = Z(A \cap E_{m-1}) + Z(A/\left(A \cap E_{m-1}\right)$ lies to the right of the path
$z_0z_1\dots z_m$ as claimed.
\end{proof}

\begin{Cor} \label{cor:HNfromfinite}
Given $E \in \cA$, assume that there are only finitely many classes
$v(A)$ of subobjects $A \subset E$ with $\Re Z(A) < \max \left\{0, \Re Z(E) \right\}$. Then $E$ admits a HN filtration.
\end{Cor}

\section{Linear algebra Lemmas}
\label{sec:linalg}

Throughout Sections \ref{sec:linalg}, \ref{sec:realvariation} and \ref{sec:Qpreserved}, we fix a quadratic form $Q$ on
$\Lambda_\R$ with:
\begin{Ass} \label{ass:nondeg}
The quadratic form $Q$ has signature $(2, \rk \Lambda - 2)$. 
\end{Ass}

\begin{Lem} \label{lem:coords}
Let $Z \colon \Lambda \to \C$ be a group homomorphism such that $Q$ is negative definite on 
$K := \Ker Z \subset \Lambda_\R$. Let $\norm{\cdot}$ be the norm on $K$ associated to $-Q|_K$, and
let $p \colon \Lambda_\R \to K$ be the orthogonal projection with respect to $Q$. 
After replacing $Z$ by a $\GL_2^+(\R)$-translate,  we have
\begin{equation} \label{eq:QandZ}
 Q(v) = \abs{Z(v)}^2 - \norm{p(v)}^2.\end{equation}
\end{Lem}
\begin{proof}
Let $K^\perp$ be the orthogonal complement of $K$. 
As $Z|_{K^\perp}$ is injective, Assumption \ref{ass:nondeg} can only hold
if $Q$ is positive definite on $K^\perp$, and  if  
$ Z|_{K^\perp} \colon K^\perp \to \C $ is an isomorphism.
Up to the $\GL_2^+(\R)$-action, we may assume this to be an
isometry. Then the the claim follows.
\end{proof}

\begin{Rem} In \cite{Bridgeland:K3}, this normalisation is used with the Mukai quadratic form replacing $Q$.
\end{Rem}

Consider the subset in $\Hom(\Lambda, \C)$ of central charges whose kernel is
negative definite with respect to $Q$; let $\cP_Z(Q)$ be its connected component containing $Z$.
\begin{Lem} \label{lem:deformincoords}
Let $Q$, $Z$ and $K$ be as in Lemma \ref{lem:coords}.
\begin{enumerate} \item \label{item:ucomplex}
For each $Z' \in \cP_Z(Q)$ there exists a unique  $g \in \GL_2^+(\R)$ and a linear map
 $u \colon K \to \C$ with $\norm{u} < 1$ such that
\[
 g Z' = Z + u \circ p. \]
\item \label{item:urealuimaginary} Up to the action of $\GL_2^+(\R)$, we can break up this deformation of $Z$ into a pure real and a purely imaginary part with analogous properties: there exist $u_\R$, $u_{i\R}$, $Z_1$ and $g_1, g_2 \in \GL_2^+(\R)$, depending continuously on $Z'$, such that
\begin{enumerate}
	\item $u_\R \colon K \to \R$ and  $u_{i\R} \colon \Ker Z_1 \to \R$ satisfy $\norm{u_\R} < 1$, $\norm{u_{i\R}} < 1$,
	\item $Z_1 := Z + u_\R \circ p$, 
	\item \label{enum:guZ1satisfiesnormalisation}
	equation \eqref{eq:QandZ} holds with $Z$ and $p$ replaced by $g_1 Z_1$ and  $p_1 \colon \Lambda_\R \to \Ker Z_1$, and
	\item \label{enum:imaginaryu}
	$g_2 Z' = g_1 Z_1 +  i u_{i\R} \circ p_1$, where $p_1 \colon \Lambda_\R \to \Ker Z_1$ is the orthogonal projection.
\end{enumerate} 
\end{enumerate}
\end{Lem}
\begin{proof}
The restriction of $Z'$ to the orthogonal complement $K^\perp$ is an isomorphism for all $Z' \in \cP_Z(Q)$. Hence there exists
a unique $g \in \GL_2(\R)$ such that $gZ'|_{K^\perp} = Z|_{K^\perp}$; since $\cP_Z(Q)$ is connected,
 in fact  $g \in \GL_2^+(\R)$.
Let  $u:= gZ'|_{\Ker Z}$; then $g Z' - Z = u \circ p$ holds both on $K$ and $K^\perp$, and
thus on all of $\Lambda_\R$.

Next, we prove $\norm{u} < 1$. Otherwise, let $v \in K$ with $\norm{v} =
1$ and $\abs{u(v)} \ge 1$, and let $\overline{v} \in K^\perp$ be such that $Z(\overline{v}) = u(v)$.
Then $Z'(v - \overline{v}) = 0$, but $Q(v - \overline{v}) = \abs{u(v)}^2 - \norm{v}^2 \ge 0$ in
contradiction to $Q$ being negative definite on $\Ker Z'$.
This completes the proof of \eqref{item:ucomplex}.

Now let $ u_\R := \Re u$, and let $Z_1$ be as above. If we write $\Lambda_\R = K \oplus K^\perp$ and
identify $K^\perp$ with $\C$ via $Z$, then $\Ker Z_1$ is the graph of $-u_\R$, and its orthogonal complement is the graph of
$\C \to K, z \mapsto \Re z \cdot u^\vee$, where $u^\vee$ corresponds to $u$ under the identification $K \cong K^\vee$ induced
by the symmetric form associated to $\norm{\cdot}$. A straightforward computation shows that
$g_1 \circ Z_1$ induces an isometry $(\Ker Z_1)^\perp \to \C$  for $g_1 := \begin{pmatrix}
\left(1 + \norm{u_\R}^2\right)^{-\frac 12} & 0 \\ 0 & 1
\end{pmatrix}\in \GL_2^+(\R)$;
by the proof of Lemma \ref{lem:coords} this implies
\eqref{enum:guZ1satisfiesnormalisation}. Moreover, $g_1 g Z' - g_1 Z_1 = g_1 \circ \Im u \circ p$ is completely imaginary. Applying part \eqref{item:ucomplex} to $g_1 g Z'$ and $g_1 Z_1$ then shows part \eqref{enum:imaginaryu}.
\end{proof}

\section{Real variations of the central charge} \label{sec:realvariation}
The key lemma, proved in this section and the next, treats the case
where only the real part of the central charge is varying:
\begin{Lem} \label{lem:realvariation}
Consider a stability condition $\sigma = (Z, \cP)$, satisfying the support property with respect to $Q$, and assume that  \eqref{eq:QandZ} holds with $p$ and $\norm{\cdot}$ as defined  in Lemma \ref{lem:coords}. Given $u \colon \Ker Z \to \R$ with
$\norm{u} < 1$, there is a stability condition $\tau = (W, \cQ)$ 
with $W = Z +  u \circ p$, satisfying the support property with
respect to $Q$, and with $d(\cP, \cQ) < \frac {\norm{u}}2$.
\end{Lem}

\begin{proof}[Proof of Theorem \ref{thm:mainthm}, assuming Lemma \ref{lem:realvariation} and
	Assumption \ref{ass:nondeg}]
Lemma \ref{lem:realvariation} automatically also applies to purely imaginary variations
of the central charge as in Lemma \ref{lem:deformincoords}.\eqref{enum:imaginaryu}, due to the  $\wGL$-action. We may assume that $Q$ and the central charge satisfy \eqref{eq:QandZ};  then Lemmas \ref{lem:deformincoords} and
\ref{lem:realvariation} combined with the $\wGL$-action give a set-theoretic section $\cP_Z(Q) \to \Stab_\Lambda(\cD)$.


Since $\GL_2^+(\R)$ acts continuously on the compact set $S^1$, there exists an open neighbourhood $B_\epsilon(1)$ of
$1 \in \wGL$ such that $\abs{\tilde g.\phi - \phi} < \epsilon$ for all $\tilde g \in B_\epsilon(1)$.  It follows that
if $(\cQ, W) = \tilde g.(\cP, Z)$, then $d(\cQ, \cP) < \epsilon$. 
Then there is an open neighborhood $U_\epsilon(Z)$ of $Z$ in $\cP_Z(Q)$
where, in the notation of Lemma \ref{lem:deformincoords}, $\norm{u_{\R}} < \epsilon, \norm{u_{i\R}} < \epsilon$, and
$g_1, g_2$ lift to $\tilde g_1, \tilde g_2 \in B_\epsilon(1)$. Then our lift $\tau = (Z, \cQ)$ of any $Z'$ in this open neighbourhood satisfies $d(\cQ, \cP) < \frac{\epsilon}{2} + \epsilon + \frac{\epsilon}{2} + \epsilon$.

By Corollary \ref{cor:continuoussection}, our section is continuous on $U_\epsilon(Z)$ for $\epsilon \le \frac 1{12}$. Since the neighbourhoods $U_{\frac 1{12}}(Z')$ for $Z' \in \cP_Z(\cQ)$ cover $\cP_Z(\cQ)$, and since the lifts constructed on each such neighbourhood  agree on the overlaps by the local injectivity of $\cZ$, this proves the Theorem.
\end{proof}

Lemma \ref{lem:realvariation} is simpler to prove since it allows (and forces) us to  leave the heart $\cA := \cP(0, 1]$ unchanged: we
will apply Proposition \ref{prop:stabviaheart} and prove
that $(\cA, W)$ produces a stability condition.
\begin{Lem} In the situation of Lemma \ref{lem:realvariation},  $W = Z + u \circ p$ is a stability
function on $\cA$.
\end{Lem}
\begin{proof} Consider $E \in \cA$; if $\Im Z(E) = \Im W(E) > 0$, there is nothing to prove.
Otherwise, $E$ is semistable with $Z(E) \in \R_{<0}$ and thus $\norm{p(E)} \le - Z(E)$.  From
 $\norm{u} < 1$ we conclude
\[
W(E) = Z(E) + u\circ p(E) \le Z(E) + \norm{u} \norm{p(E)}  < Z(E) - Z(E) = 0. \qedhere
\] 
\end{proof}

We will use Proposition \ref{prop:HNviapoly} and Corollary \ref{cor:HNfromfinite} to prove that $(\cA, W)$ satisfies the HN property.

Let us define the \emph{mass} $m^Z(E)$ of $E \in \cA$ with respect to $Z$ as the length of the boundary
of $\HN^Z(E)$ on the left between $0$ and $Z(E)$.

\begin{Lem} \label{lem:boundZ}
For all $E \in \cA$ we have
$\norm{p(E)} \le m^Z(E)$.
\end{Lem}
\begin{proof}
If $E$ is semistable, then
$0 \le Q(E) = \abs{Z(E)}^2 - \norm{p(E)}^2
= \left(m^Z(E)\right)^2 - \norm{p(E)}^2$, which is exactly the claim. 
Otherwise, consider the HN filtration $E_0 \into E_1
\into \dots \into E_m = E$. Combined with the triangle inequality, this gives
\[
\norm{p(E)} \le \sum_i \norm{p(E_i/E_{i-1})} \le \sum_i \abs{Z(E_i/E_{i-1})}
= \sum_i \abs{Z(E_i) - Z(E_{i-1})} = m^Z(E). \qedhere
\]
\end{proof}

The following Lemma needs no proof:
\begin{Lem} \label{lem:subHNpoly}
If $A \subset E$, then $\HN^Z(A) \subset \HN^Z(E)$.
\end{Lem}

\begin{Lem} \label{lem:boundlength}
Given any subobject $A \subset E$, we have
\[ m^Z(A) - \Re Z(A) \le m^Z(E) - \Re Z(E). \]
\end{Lem}
\begin{proof}
This follows from the previous Lemma, convexity and a picture, see
fig.~\ref{fig:bound-subobject}. Indeed, choose $x > \Re Z(A), \Re Z(E)$; let $a = x + i \Im Z(A)$ and
$e = x + i \Im Z(E)$. Let $\gamma_A$ be the path that follows the boundary of $\HN^Z(A)$ from $0$ to
$Z(A)$, and then continues horizontally to $a$; similarly $\gamma_E$ follows the boundary of
$\HN^Z(E)$ and then continues to $e$. Their lengths are given as
\[ \abs{\gamma_A} = m^Z(A) + x - \Re Z(A), \quad \abs{\gamma_E} = m^Z(E) + x - \Re Z(E).
\]
On the other hand, convexity and Lemma \ref{lem:subHNpoly} imply $\abs{\gamma_A} \le
\abs{\gamma_E}$; for example, if $\gamma_I$ denotes the intermediate path that follows the boundary
of $\HN^Z(E)$ up to height $\Im Z(A)$ and then goes horizontally to $a$, we clearly have
$\abs{\gamma_A} \le  \abs{\gamma_I} \le  \abs{\gamma_E}$.
\end{proof}

\begin{figure}
\centering
\begin{minipage}{0.45\textwidth}
\centering
        \includegraphics{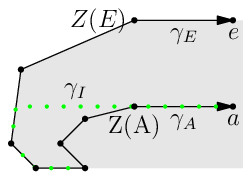}
    \captionof{figure}{Proof of Lemma \ref{lem:boundlength}}
    \label{fig:bound-subobject}
\end{minipage}
\begin{minipage}{0.45\textwidth}
	\centering
        \includegraphics{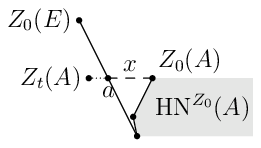}
    \captionof{figure}{Proof of Lemma \ref{lem:bounddistance}}
    \label{fig:bounddistance}
\end{minipage}
\end{figure}

\begin{Lem}\label{lem:finitesubobjectsrealpart} Given $C \in \R$, the set of  $v(A) \in \Lambda$ for subobjects $A \subset E$ with
	$\Re W(A) < C$ is finite.
\end{Lem}
\begin{proof}
Given any such $A$, we 
use Lemmas \ref{lem:boundZ} and \ref{lem:boundlength} to obtain
\begin{align*}
C  > \Re W(A) \ge \Re Z(A) - \norm{u} \norm{p(A)}
& > (1 - \norm{u})\Re Z(A) - \norm{u} \left(m^Z(A) - \Re Z(A)\right) \\
& \ge (1 - \norm{u})\Re Z(A) - \norm{u} \left(m^Z(E) - \Re Z(E)\right).
\end{align*}
Since $\norm{u} < 1$, this bounds $\Re Z(A)$ from above.
On the other hand, $Z(A) \in \HN^Z(E)$, and thus $Z(A)$ is constrained to lie in a compact region of $\C$.
Using Lemmas \ref{lem:boundlength} and \ref{lem:boundZ} again, this gives an upper bound first for $m^Z(A)$
and consequently for $\norm{p(A)}$. Hence $v(A)$ is contained in a compact
region of $\Lambda \otimes \R$ depending only on $E$ and $C$, and the claim follows.
\end{proof}

Lemma \ref{lem:finitesubobjectsrealpart} and Corollary \ref{cor:HNfromfinite} imply the existence of HN filtrations for
$W$ on $\cA$, and thus yield a pre-stability condition
$\tau = (\cA, W)$ by Proposition \ref{prop:stabviaheart}; write $\cQ$ for the associated slicing.
\begin{Lem} \label{lem:bounddistance}
The pre-stability condition $\tau = (W, \cQ)$ satisfies
$d(\cP, \cQ) < \frac {\norm{u}}2$.
\end{Lem}
\begin{proof}
	We apply Lemma \ref{lem:generaldistanceboundfromZ}: by construction, by \eqref{eq:QandZ}, and by the
	assumption $\norm{u} < 1$ we have
		\[ 
		\frac{\abs{W(v)-Z(v)}}{\abs{Z(v)}} = \frac{\abs{u \circ p(v)}}{\abs{Z(v)}} \le \norm{u} \frac{\norm{p(v)}}{\abs{Z(v)}} \le \norm{u} < \sin \pi\frac{\norm{u}}{2} \ \text{for $v \in \Lambda$ with $Q(v) \ge 0$.} \qedhere
	\]
\end{proof}

\section{The support property is preserved}
\label{sec:Qpreserved}

It remains to show that $(\cA, W)$ satisfies the support property
with respect to $Q$, i.e.~$Q(v(E)) \ge 0$ for all $W$-stable $E \in \cA$. The basic reason is that this inequality preserved by wall-crossing:

\begin{Lem} \label{lem:Qpreserved}
Let $\sigma = (Z, \cP)$ be pre-stability condition.
Assume that $Q$ is a non-degenerate quadratic form on $\Lambda_\R$ of signature $(2, \rk \Lambda-2)$ such
that $Q$ is negative definite on $\Ker Z$. If $E$ is strictly $\sigma$-semistable and admits a Jordan-H\"older filtration with factors $E_1, \dots, E_m$, and if $Q(v(E_i)) \ge 0$ for $i = 1, \dots, m$, then
$Q(v(E)) \ge 0$.
\end{Lem}
\begin{proof}
We apply Lemma \ref{lem:coords}; then $Q(v) \ge 0$ is equivalent to $\abs{Z(v)} \ge \norm{p(v)}$. We
obtain
\[ \abs{Z(E)} = \sum_i \abs{Z(E_i)} \ge \sum_i \norm{p(v(E_i))} \ge \norm{\sum_i p(v(E_i))} = \norm{p(v(E))} \]
where the first equality holds since the central charges of all $E_i$ are aligned, the first inequality holds by assumption, and the second inequality is the triangle inequality.
\end{proof}

The proof strategy is thus clear: we use wall-crossing for the path of stability functions $Z_t = Z + t\cdot u \circ p$ on $\cA$, for $0 \le t \le 1$. If $E \in \cA$ is $Z_1$-stable with $Q(v(E)) < 0$, then it must be
$Z_0$-unstable; wall-crossing gives a $t \in [0, 1)$ such that $E$ is strictly $Z_t$-semistable; by the Lemma, one of its Jordan-H\"older factors will also violate the inequality, and we proceed by induction. To conclude, we have to show that we can find such a wall, and that this process terminates.

\begin{Lem} \label{lem:duh}
Given two objects $A, E \in \cA$, denote their phases with respect to $Z_t$ by
$\phi^t(A), \phi^t(E)$, respectively. If the set of $t \in [0,1]$ with $\phi^t(A) \ge \phi^t(E)$ is
non-empty, then it is a closed subinterval of $[0,1]$ containing one of its endpoints.
\end{Lem}
\begin{proof}
The condition is equivalent to $\frac{-\Re Z_t(A)}{\Im Z_t(A)} \ge \frac{-\Re Z_t(E)}{\Im Z_t(E)}$,
which is a linear inequality in $t$.
\end{proof}

		\begin{wrapfigure}[8]{r}{0.3\textwidth}
  \centering
        \includegraphics{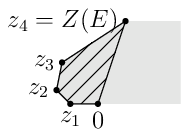}
    \caption{The truncated HN polygon}
    \label{fig:HNpolytruncated}
\end{wrapfigure}
Consider the polygon whose vertices are the extremal points of $\HN^{Z_0}(E)$ on the left; we will call this the
\emph{truncated HN polygon of $E$}, see fig.~\ref{fig:HNpolytruncated}. Note that if $A \subset E$ is a
subobject with $\phi^0(A) \ge \phi^0(E)$, then $Z_0(A)$ is contained in the truncated HN polygon of
$E$; by Lemmas \ref{lem:boundlength} and \ref{lem:boundZ} there are only finitely
many classes $v(A)$ of such subobjects.

\begin{Lem}
Every $Z_1$-semistable object $E \in \cA$ satisfies $Q(E) \ge 0$.
\end{Lem}
\begin{proof}
Otherwise, $E$ must be $Z_0$-unstable. By Lemma \ref{lem:duh} and the following observation, there are only finitely many classes
$v(A)$ of subobjects $A \into E$ that destabilise $E$ with respect $Z_t$ for any $t\in [0,1]$. Hence
there is a wall $t_1 \in (0, 1]$ such
that $E$ is strictly semistable with respect to $Z_{t_1}$, and moreover $E$ admits a Jordan-H\"older filtration with respect to $Z_{t_1}$. By Lemma 
\ref{lem:Qpreserved}, there are subobjects $G_1 \into F_1 \into E$ of the same phase, such that $F_1/G_1$ is $Z_{t_1}$-\emph{stable} with
$Q(v(F_1/G_1)) < 0$.

Applying the same logic to $F_1/G_1$, we obtain $t_2 \in (0, t_1)$ and subobjects $G_1 \subset G_2 \subset F_2 \subset F_1 \subset E$
such that $F_2/G_1, G_2/G_1$ and $F_1/G_1$ all have the same phase with respect to $t_2$, and such that
$Q(v(F_2/G_2)) < 0$. Continuing by induction, we obtain a sequence $t_1 > t_2 > t_3 > \dots$ in $(0,1)$ and  chains of subobjects $G_1 \subset G_2 \subset G_3 \subset \dots \subset E$ and $E \supset F_1 \supset
F_2 \supset F_3 \supset \dots$

Lemma \ref{lem:duh} gives $\phi^{t_2}(F_1) \ge \phi^{t_2}(E)$ and $\phi^{t_2}(G_1) \ge
\phi^{t_2}(E)$. Since  $Z_{t_2}(F_2)$ lies on the line segment
connecting $Z_{t_2}(F_1)$ and $Z_{t_2}(G_1)$, we also have $\phi^{t_2}(F_2) \ge \phi^{t_2}(E)$ (and
therefore $\phi^{t}(F_2) \ge \phi^t(E)$ for all $t \in [0, t_2]$);
similarly for $G_2$. Continuing
by induction, we conclude $\phi^0(F_i) \ge \phi^0(E)$. This is a contradiction to the
observation above that there are only finitely many classes $v(A)$ of subobjects $A \subset E$ with
$Z_0(A)$ lying in the truncated HN polygon.
\end{proof}

This concludes the proof of Lemma \ref{lem:realvariation}, and thus of Theorem \ref{thm:mainthm} under the
Assumption \ref{ass:nondeg}.

\section{Reductions}

Finally, we will show that we can always reduce the situation to the case where Assumption \ref{ass:nondeg} holds. 
By abuse of language, we call a quadratic form degenerate or non-degenerate if the associated symmetric bilinear form
is degenerate or non-degenerate, respectively.

\begin{Lem}
Assume that the quadratic form $Q$ on $\Lambda_\R$ is degenerate. Then there exists an injective map $\Lambda_\R \into \overline{\Lambda}$ of real vector spaces and a non-degenerate quadratic form $\overline{Q}$ on $\overline{\Lambda}$, extending $Q$, such that any central charge
$Z \colon \Lambda_\R \to \C$ with kernel negative definite with respect to $Q$ extends to a central charge
$\overline{Z} \colon \overline{\Lambda} \to \C$ with kernel negative definite with respect to $\overline{Q}$.
\end{Lem}
\begin{proof}
Let $N \into \Lambda_\R$ be the null space of $Q$; we will only treat the case $\dim_\R N = 1$
(otherwise, we can iterate the construction that follows). Choose a splitting $\Lambda_\R \cong N
\oplus C$; then for $n \in N, c \in C$, we have $Q(n \oplus c) = Q(c)$. Let
$\overline{\Lambda_\R} := N \oplus N^\vee \oplus C$, let $q$ be the canonical quadratic form on
the hyperbolic plane $N \oplus N^\vee$, and set $\overline{Q} := q \oplus Q|_C$. 

Given $Z$ as above, the restriction $Z|_N$ is  injective, and we may
assume that $Z$ maps $N$ to the real line. Let $n \in N$ be such that $Z(n) = 1$, and let
$n^\vee \in N^\vee$ be the dual vector with $(n, n^\vee) = 1$. We claim that for $\alpha \gg 0$,
the extension of $Z$ defined by $Z'(n^\vee) = \alpha$ has the desired property.

Let $K := \Ker Z$; then the kernel of $Z'$ is contained in $N \oplus N^\vee \oplus K$, and
given by vectors of the form $a \cdot n - \frac{a}{\alpha}\cdot n^\vee + k$ for $k \in K, a \in \R$.
For such vectors, we have
\[
 Q\left(a\cdot n  - \frac{a}{\alpha}\cdot n^\vee + k \right)
= - \frac{2a^2}{\alpha} - \frac{2a}{\alpha}(n^\vee, k) + Q(k).
\]
This is a quadratic function in $a$ with negative constant term; its discriminant is negative if
\[ \alpha > \max \stv{\frac{(n^\vee, k)^2}{-Q(k)}}{k \in K, k \neq 0} \]
(which is finite since $-Q(\cdot)$ is a positive definite form on $K$).
\end{proof}

Replacing $\Lambda$ by $\Lambda \oplus \Z$ and $v$ by 
\[ K(\cD) \xrightarrow{v} \Lambda \into \Lambda \oplus \Z \]
we can therefore restrict to the case where $Q$ is non-degenerate: given a path $Z_t$ of central
charges in $\Hom(\Lambda_\R, \C)$ that are negative definite with respect to $Q$, we can choose
extensions $\overline{Z}_t$ as in the Lemma that form a continuous path in $\Hom(\overline{\Lambda}, \C)$.
If we can lift the latter path to a path of stability conditions
$\overline{\sigma}_t = (\overline{Z}_t, \cP_t)$ that satisfy the support property 
with respect to $\overline{Q}$, then $\sigma_t := (Z_t, \cP_t)$ is a
path of stability conditions satisfying the support property with respect to $Q$.
The reduction to the case where $Q$ has signature
$(2, \rk \Lambda -2)$ works similarly:

\begin{Lem} Assume that $Q$ is non-degenerate and of signature $(p, \rk \Lambda -p)$ for $p \in \{
0, 1\}$. 
Let $\overline{\Lambda} := \Lambda_\R \oplus \R$, and let $\overline{Q}$ be the extension given by
$\overline{Q}(v, \alpha) = Q(v) + \alpha^2$ for $v \in \Lambda_\R$ and $\alpha \in \R$. 
Then any central charge $Z$ on $\Lambda_\R$ whose kernel is negative definite with respect to $Q$ extends to
a central charge $\overline{Z}$ on $\overline{\Lambda}$ whose kernel is negative definite with respect to
$\overline{Q}$. 
\end{Lem}
\begin{proof}
We claim that there exists $z \in \C$ such that for all $v \in \Lambda_\R$ with $Z(v) = z$, we have
$Q(v) < -1$. Indeed, 
let $K \subset \Lambda_\R$ be the kernel of $Z$, and let $K^\perp$ be its orthogonal complement. Then clearly
we may assume $v \in K^\perp$. Since the restriction of $Z$ to $K^\perp$ is injective, and since
$K^\perp$ either has rank one, or has signature $(1, -1)$ with respect to $Q$, the claim is evident.

Using the claim, we can set $\overline Z(v, \alpha) := Z(v) + \alpha z$. 
\end{proof}
This concludes the proof of Theorem \ref{thm:mainthm}. 

\section{Application} \label{sec:application}

\begin{proof}[Proof of Corollary \ref{cor:P0covering}]
Using the same arguments as in the previous section, we may assume that the Mukai pairing
on $\Lambda$ has signature $(2, \rk \Lambda -2)$. 

By Serre duality, any $\sigma$-stable object $E \in \cD$ satisfies $\Hom(E, E[i]) = 0$ for
$i < 0$ or $i > 3$ and $\Hom(E, E) = \C = \Hom(E, E[2])$; therefore, $(v(E), v(E)) \ge -2$.
Moreover, Serre duality induces a non-degenerate symplectic form on $\Ext^1(E, E)$, and it
has even dimension; thus $v(E)$ is a root, or $(v(E), v(E)) \ge 0$.

Let $\sigma = (Z, \cP)$ be a stability condition with $Z \in \cP_0(\cD)$. As in 
Lemma \ref{lem:coords} we may assume
\[ (v, v) = \abs{Z(v)}^2 - \norm{p(v)}^2, \]
where $p \colon \Lambda_\R \to \Ker Z$ is the orthogonal projection onto the kernel of $Z$, and
where $\norm{\cdot}$ denotes the norm on $\Ker Z$ induced by the negative of the Mukai pairing.
We claim that  
\begin{equation} 
C:= \inf \stv{\abs{Z(\delta)}}{\delta \in \Lambda, (\delta, \delta) = -2} > 0.
\end{equation}
Indeed, if $\abs{Z(\delta)} \le 1$, then $\norm{p(\delta)} \le \sqrt{3}$;
as $\abs{Z(\cdot)} + \norm{p(\cdot)}$ is a norm on $\Lambda_\R$,
there are only finitely many integral classes satisfying both inequalities.
Since $Z(\delta) \neq 0$ by assumption, the claim follows.

Now set 
\[ Q(v) := (v, v) + \frac 2{C^2} \abs{Z(v)}^2. \]
Clearly $Q$ is negative definite on $\Ker Z$ and depends only on $Z$. Moreover, $Q(\delta) \ge 0$ for all roots $\delta$,
and $Q(v) \ge 0$ for all classes with $(v, v) \ge 0$; therefore, any stability condition $\sigma' = (Z', \cP')$ with 
$Z' \in \cP_Z(Q)$ satisfies the support property with respect to $Q$.

Theorem \ref{thm:mainthm} shows that the restriction of $\cZ$ to the preimage of $\cP_Z(Q)$ is a covering of $\cP_Z(Q)$. Since the neighbourhoods $\cP_Z(Q)$ cover $\cP_0(\cD)$, this completes the proof.
\end{proof}

\bibliography{all}                      
\bibliographystyle{halpha}     

\end{document}